\documentclass[12pt]{amsart}

\numberwithin{equation}{section}

\usepackage{amsmath,amsthm,amsfonts,eucal,}
\usepackage{xypic}
\usepackage{graphicx}
\usepackage{amssymb}
\usepackage{bbm}

\def\cb{{\mathcal B}}

\def\ce{{\mathcal E}}

\def\ch{{\mathcal H}}

\def\cs{{\mathcal S}}

\def\cu{{\mathcal U}}

\def\ga{{\mathfrak A}} 
\def\gb{{\mathfrak B}}

\def\bc{{\mathbb C}}

\def\bn{{\mathbb N}}
\def\bo{{\mathbb O}}
\def\bp{{\mathbb P}}
\def\br{{\mathbb R}}
\def\bt{{\mathbb T}}

\def\bz{{\mathbb Z}}

\def\a{\alpha}

\def\r{\rho}
\def\s{\sigma}

\def\om{\omega} \def\Om{\Omega}

\newtheorem{thm}{Theorem}[section]
\newtheorem{example}[thm]{Example}
\newtheorem{lem}[thm]{Lemma}
\newtheorem{cor}[thm]{Corollary}
\newtheorem{prop}[thm]{Proposition}

\theoremstyle{definition}
\newtheorem{rem}[thm]{Remark}

\def\Im{\mathop{\rm Im}}

\begin{document}
\title[Freedman's theorem for unitarily invariant states on the CCR algebra]
{Freedman's theorem for unitarily invariant states on the CCR algebra}

\author{Vitonofrio Crismale}
\address{Vitonofrio Crismale\\
Dipartimento di Matematica\\
Universit\`{a} degli studi di Bari\\
Via E. Orabona, 4, 70125 Bari, Italy}
\email{\texttt{vitonofrio.crismale@uniba.it}}

\author{Simone Del Vecchio}
\address{Simone Del Vecchio\\
Dipartimento di Matematica\\
Universit\`{a} degli studi di Bari\\
Via E. Orabona, 4, 70125 Bari, Italy}
\email{\texttt{simone.delvecchio@uniba.it}}

\author{Tommaso Monni}
\address{Tommaso Monni\\
Dipartimento di Matematica\\
Universit\`{a} degli studi di Bari\\
Via E. Orabona, 4, 70125 Bari, Italy}
\email{\texttt{tommaso.monni@uniba.it}}

\author{Stefano Rossi}
\address{Stefano Rossi\\
Dipartimento di Matematica\\
Universit\`{a} degli studi di Bari\\
Via E. Orabona, 4, 70125 Bari, Italy}
\email{\texttt{stefano.rossi@uniba.it}}

\begin{abstract}

The set of states on ${\rm CCR}(\ch)$, the CCR algebra
of a separable Hilbert space $\ch$, is here looked at as a natural object
to obtain a non-commutative version of Freedman's theorem for unitarily invariant stochastic processes.
In this regard, we provide a complete description of the compact convex set of  states of
${\rm CCR}(\ch)$ that are invariant under the action of all automorphisms induced in second quantization by
unitaries of $\ch$. We prove that this set is a Bauer simplex, whose extreme states are either
the canonical trace of the CCR algebra or Gaussian states  with variance at least $1$.


\vskip0.1cm\noindent \\
{\bf Mathematics Subject Classification}:  60G09, 46L53, 81V73, 45A55\\
{\bf Key words}: CCR algebras, invariant states, Freedman's theorem, regular states
\end{abstract}

\maketitle
\section{Introduction}

A sequence of real-valued random variables
is said to be orthogonally invariant (or rotatable, as in \cite{Ka}) if the joint
distribution of any finite subset of variables taken from the sequence
has spherical symmetry. Rotatability is among the strongest distributional symmetries that a stochastic process
may enjoy.
Indeed, it is a classical result due to Freedman \cite{Freedman} that the joint distribution
of an orthogonally invariant sequence is a mixture ({\it i.e.} convex combination) of measures each of which
is an infinite product of a single centered Gaussian distribution. Furthermore,
extreme measures correspond to sequences of independent, identically distributed variables
with centered Gaussian distribution with arbitrary variance.
A weaker distributional symmetry, exchangeability only requires invariance of the finite-dimensional
distributions under permutations. Exchangeable sequences are ruled by de Finetti's theorem, which says
that their joint distribution is still a mixture, and that  extreme measures are again infinite products. The
common distribution, however, remains undetermined, see {\it e.g.} \cite{Ka}.
Unlike the classical framework, in most of the commonly used
non-commutative models  already the structure of the convex of exchangeable (also known as symmetric) states is too
poor to allow for a strict inclusion of the convex of orthogonally invariant states in the set of symmetric states.
For instance, the only symmetric state on the so-called $q$-deformed $C^*$-algebra with $|q|<1$ ($q=0$ corresponds to the free case) is the Fock vacuum, see
{\it e.g.} \cite{CRAMPA}, whereas symmetric states on both Boolean and monotone $C^*$-algebra make up a mere segment, see
\cite{CrFid, CFG}. The case of  the CAR algebra  is the odd one out: although it does feature a rich convex set of symmetric states, thoroughly analyzed  in \cite{CFCMP}, its
rotatable states have nonetheless been proved to be the same
as symmetric states in \cite{CDR}.
Far-reaching general results can be stated all the same, at the cost, though, of
working with symmetries implemented by the natural coaction of the quantum orthogonal/unitary groups on the $*$-algebra of
non-commutative polynomials. Doing so, the Gaussian distribution must be replaced by free operator-valued circular distribution,
see \cite{Curran}.\\
The motivation of the present paper, however, is to exhibit a natural model where rotatable states are  fewer than
symmetric states and may be described in a meaningful way.
The CCR algebra associated with a separable complex Hilbert space $\ch$, ${\rm CCR}(\ch)$,
turns out to provide a model  particularly suited to this end for three main reasons.
First, it is naturally acted upon by unitary transformations and permutations, which puts one in a position to consider unitarily invariant and symmetric states. Second, it contains a subalgebra $\gb_0$, which is still a CCR algebra, and, more importantly,
is an infinite tensor product to which classical results of Størmer in \cite{StorJFA69} apply, hence
 making the structure of exchangeable states
completely known. Third, it also has the merit of providing a setting where
a genuine generalization of Freedman's classical theorem may be stated, see Remark \ref{classical}.\\
It is worth remarking, however, that the actual quantum random variables associated with the CCR algebra are  in fact
given by a sequence of pairs 
$\{(Q_i, P_i): i\in\bz\}$ of non-commuting self-adjoint operators, each of which should be conceived as a complex variable.
Therefore, unitary invariance appears the natural symmetry to take into account
as opposed to orthogonal invariance.\\
With this in mind, we start our analysis  by considering
our subalgebra $\gb_0\subset {\rm CCR}(\ch)$  generated by Weyl operators $W(x)$, with $x$ running through
the  dense subspace $\ch_0$ of the Hilbert space $\ch$  of those vectors with only finitely many  Fourier coefficients different from $0$ with respect to a given basis. Working with $\gb_0$  is particularly convenient because
of its factoring into an infinite tensor product of ${\rm CCR}(\bc)$ with itself,
as well as being
invariant under the action of  both $\bo_\bz$, the group of localized orthogonal transformations, and
$\mathbb{U}_\bz$, the group of localized unitary transformations. As a benefit of the tensor structure of $\gb_0$, the restricted dynamics is asymptotically abelian, from which we derive a number
of consequences. In particular, both the sets of $\bo_\bz$-invariant states and of $\mathbb{U}_\bz$-invariant states on $\gb_0$ are  Choquet simplices, and their extreme states are infinite products of a single state $\varphi$ on ${\rm CCR}(\bc)$, Proposition \ref{extreme}, in analogy
with what happens in the classical setting.
Moreover, knowing what the extreme states of our simplices look like proves key to
showing that each of our simplices of invariant states is contained in the larger ones
as a face, Proposition \ref{faces}.\\
The characteristic function of the state $\varphi$ above, that is $\bc\ni z\mapsto \varphi(W(z))\in\bc$, must satisfy a certain functional equation, the solutions of which
are Gaussians of the type $\varphi(W(z))=e^{-\sigma^2(\arg z) |z|^2}$, where the variance $\sigma^2$ only depends on the phase of $z$, and is also allowed to take  possibly infinite values, Lemma \ref{allsolutions}.
As long as rotatable states are focused on, the variance may and will vary with the phase, as shown in Example \ref{rota}.
 As soon as one requires unitary invariance, though, the dependence of $\sigma^2$ on the phase disappears. More precisely, there are only two possibilities. In the first, the characteristic function is
 $\varphi(W(z))=\mathbbm{1} _{\{0\}}(z)$, the indicator function of the singleton $\{0\}$.
In the second,  the characteristic function is Gaussian, namely $\varphi(W(z))=e^{-\frac{1}{2}\sigma^2 |z|^2}$, where $\sigma^2$ cannot be chosen as small as one wishes. In fact, Heisenberg's uncertainty principle drives
$\sigma^2$ to be at least $1$, Proposition \ref{unitary}.
Phrased differently, all regular extreme $\mathbb{U}_\bz$-invariant states on $\gb_0$ are quasi-free states of the type $\om_{\s^2}(W(x))= e^{-\frac{\s^2}{2}\|x\|^2}$, $x\in\ch_0$, with
$\s^2\geq 1$. In addition, taking $\sigma^2$ as small as possible, {\it i.e. }$\sigma^2=1$, yields the only pure $\mathbb{U}_\bz$-invariant state, which is nothing but the Fock vacuum.
Along the way, we also show that our techniques lead us to a new proof of Freedman's theorem based on non-commutative ergodic theory for $C^*$-algebras, Remark \ref{classical}, by restricting the dynamics to a suitable commutative
subalgebra, which is the infinite tensor product of ${\rm CCR}(\br)$ with itself.
In addition, states on the whole $\gb_0$ lend themselves to being interpreted as quantum stochastic processes with sample algebra given by
${\rm CCR}(\bc)$, the CCR algebra of a single particle.\\
Once $\mathbb{U}_\bz$-invariant states on $\gb_0$ have been dealt with, we draw our attention to
the whole ${\rm CCR}(\ch)$. Now $\mathbb{U}_\bz$-invariance is too weak a symmetry to determine
the structure of the resulting states, not least because the map that restricts  $\mathbb{U}_\bz$-invariant
states of ${\rm CCR}(\ch)$ to $\gb_0$ fails to be injective, as a drawback of the existence of the so-called
flat extensions, Remark \ref{flat}. By considering the larger symmetry group $\mathcal{U}(\ch)$ of all unitaries
on $\ch$, though, we manage to come to a complete description of the compact convex set of
the corresponding invariant states, Theorem \ref{Choquet1}. This turns out to be a Bauer simplex
(a Choquet simplex whose extreme points make up a compact set) whose extreme
points are either quasi-free Gaussian states of the type $\om_{\s^2}(W(x))=e^{-\frac{\s^2 \|x\|^2}{2}}$, $x\in\ch$, with $\s^2\geq 1$,
or the canonical trace $\tau$ of  ${\rm CCR}(\ch)$ defined by $\tau(W(x))=0$ for all non-zero $x$ in $\ch$.
In particular, the canonical trace is thus seen to be the only extreme $\mathcal{U}(\ch)$-invariant state on
${\rm CCR}(\ch)$ which is not regular.

\section{Preliminaries}\label{prel}

Let $(\ch,\sigma)$  be a symplectic real space, that is a pair where $\ch$ is a real vector space and
$\sigma:\ch\times\ch\rightarrow\br$ an antisymmetric bilinear form.
We recall that associated with any such pair there is a CCR algebra, denoted by ${\rm CCR}(\ch,\sigma)$,
which is by definition the universal $C^*$-algebra generated by the family of elements
$\{W(x)\,|\,x\in \ch\}$ satisfying the so-called canonical commutation rules in the Weyl form

\begin{equation}\label{CCRWeyl}
\begin{aligned}
	&W(x)^*=W(-x)\\
	&W(x)W(y)=e^{i\sigma(x,y)} W(x+y)
\end{aligned}
\end{equation}
for all $x, y\in\ch$.
Details on how this $C^*$-algebra may be shown to exist are found in {\it e.g.}\cite{Petz}.
In particular, ${\rm CCR}(\ch,\sigma)$ is a unital $C^*$-algebra with unit $W(0)$, and
$W(x)$ is a unitary element for every $x\in\ch$.\\
If we in addition suppose that $\s$ is non-degenerate ($x\in\ch$ with
$\s(x, y)=0$ for all $y\in\ch$ implies $x=0$), then ${\rm CCR}(\ch,\sigma)$ turns out to be
a simple $C^*$-algebra, see {\it e.g.}  \cite{Petz}.\\
Throughout the present paper we will be working with $\ch$ being given by
the real vector space associated with a (separable) complex Hilbert space
$\ch$, thought of as a symplectic space with canonical symplectic form
$$\s(x, y):=\Im \langle x, y\rangle\,\,\, \textrm{for all}\,\, x, y\in\ch,$$ where
$\langle\cdot, \cdot\rangle$ denotes the inner product of $\ch$ (assumed linear in the first variable and anti-linear in the second). As is well known, ${\rm CCR}(\ch, \s)$ admits a (faithful) remarkable representation, the Fock representation, whose construction
we briefly sketch for the reader's convenience.
For each integer $n\geq 0$, we denote by
$\ch^{\otimes^n}$ the tensor product of $\ch$ with itself $n$ times, with the convention that
$\ch^{\otimes^0}$ is a one-dimensional space, which we denote by $\bc\Om$.
The full Fock space is then defined as the infinite direct sum below
$$\mathfrak{F}(\ch):=\bigoplus_{n=0}^\infty\ch^{\otimes^n}\,,$$
and $\Om$ is referred to as the vacuum vector in the literature.\\
We next consider the orthogonal projection $P_+$ of $\mathfrak{F}(\ch)$ uniquely determined by
$$P_+\,x_{1}\otimes\cdots\otimes x_{n}:= \frac{1}{n!}\sum_{s \in S_n} x_{s(1)}\otimes\cdots\otimes x_{s(n)}\,,$$
with $S_n$ the group of all permutations of $\{1, \ldots, n\}$.
The symmetric (or Bose) Fock space is then defined as the range of the projection $P_+$, that is
$$\mathfrak{F}_+(\ch):=P_+(\mathfrak{F}(\ch))\, .$$
With any $x\in\ch$ it is possible to associate a vector $e(x)$ in $\mathfrak{F}_+(\ch)$, known as the exponential vector of $x$, defined as
$$e(x):=\bigoplus_{n=0}^\infty \frac{1}{n!} x^{\otimes^n}\,.$$
We continue our quick overview of the Bose Fock space by recalling that, for any given $A\in\cb(\ch)$, it is possible
to define a (possibly unbounded) operator $\Gamma(A)$ on $\mathfrak{F}_+(\ch)$, the second quantization of $A$, whose action on simple tensors is
$$\Gamma(A) (x_1\otimes\cdots\otimes x_n):= Ax_1\otimes\cdots \otimes Ax_n\,,  \quad n\in\bn,\, x_1, \ldots, x_n\in\ch\,.$$
Accordingly, $\Gamma(A)$ sends exponential vectors to exponential vectors, for
$$\Gamma(A) e(x)= e(Ax)\,,\quad x\in\ch\,. $$
Furthermore, one has $\Gamma(A^*)=\Gamma(A)^*$ and  $\Gamma(AB)=\Gamma(A)\Gamma(B)$ for all $A, B\in\cb(\ch)$, hence
$\Gamma(U)$ is unitary if $U$ is so.\\
For every $x\in\ch$, define on $\mathfrak{F}(\ch)$ a pair of unbounded densely-defined operators  $a(x)$ and $a^\dag(x)$, known as full
annihilator and creator, respectively. For all $n\in\bn$ and $x_1, \ldots, x_n$ in $\ch$, set

\begin{align*}
&a(x)\Om=0\\
&a(x) (x_{1}\otimes\cdots\otimes x_{n}):=n^{\frac{1}{2}}\langle x_1, x \rangle x_{2}\otimes\cdots\otimes x_{n}
\end{align*}
and
\begin{align*}
&a^\dag(x)\Om=x\\
&a^\dag(x) (x_{1}\otimes\cdots\otimes x_{n})=(n+1)^{\frac{1}{2}} x\otimes x_{1}\otimes\cdots\otimes x_{n}\, .
\end{align*}
At this point, for every $x\in\ch$ Bose annihilators and creators can be simply defined as
\begin{align*}
&a_+(x):= P_+ a(x) P_+\\
&a^\dagger_+(x):= P_+ a^\dagger (x) P_+\,.
\end{align*}
Details on the domains of Bose annihilators and creators are to be found in {\it e.g.} \cite{BR2}.
For every $x\in\ch $, the sum $a_+(x)+a^\dagger_+(x)$ turns out to be an essentially self-adjoint operator,
whose unique self-adjoint extension is denoted by $\Phi(x)$, which is called the field operator.
By exponentiating the field operators, we finally arrive at the following family of unitary operators
$$W'(x):=e^{i\Phi(x)}\,, \quad x\in\ch$$
which can be seen to satisfy the CCR relations \eqref{CCRWeyl}.\\
Therefore, by universality of the CCR algebra there exists a unique $*$-representation of
${\rm CCR}(\ch, \s)$ that sends $W(x)$ to $W'(x)$, for all $x\in\ch$. This is precisely
the Fock representation of ${\rm CCR}(\ch, \s)$ and is known to be irreducible.
The vector state associated with $\Om$, the vacuum vector, is thus a pure state, known as the Fock vacuum, which  is
still denoted with $\Om$ by a slight abuse of notation.
Note that $\Om(W(x))=e^{-\frac{1}{2}\|x\|^2} $.\\
The Fock representation is the most important example of a regular representation of ${\rm CCR}(\ch, \s)$, by which we mean a representation
$\pi$ such that for every $x\in \ch$  the one-parameter unitary group  $\{\pi(W(tx)): t\in\br\}$
is strongly continuous. A state $\om$ on ${\rm CCR}(\ch, \s)$ is accordingly said to be regular if its
GNS representation is regular.\\

Our next goal is to define quite a vast class of endomorphisms (automorphisms) on CCR algebras.
To this end, let  $O$ be any real operator on $\ch$
preserving $\s$, that is $\sigma(Ox, Oy)= \sigma(x, y)$ for all $x, y$ in $\ch$.
Then it is immediate to see that the set of unitaries $\{W(Ox)\,|\,x\in \ch)\}$ continues to satisfy
the canonical commutation rules. Therefore, by universality of ${\rm CCR}$ algebras
 there must exist a  $*$-monomorphism $\rho_O$ of ${\rm CCR}(\ch,\sigma)$ uniquely determined
by
\begin{equation}\label{action}
\rho_O(W(x))=W(Ox)\,, \quad x\in\ch\,.
\end{equation}
Since the symplectic form we are working with is the imaginary part of the canonical Hermitian product
of $\ch$, for any isometry $S$ of $\ch$ there exists a $*$-monomorphism
$\rho_S$ of ${\rm CCR}(\ch,\sigma)$ such that
$$\rho_S(W(x))=W(Sx)\,, \quad x\in\ch\,.$$
Furthermore, $\rho_{S_1S_2}=\rho_{S_1}\circ \rho_{S_2}$ for all isometries
$S_1, S_2$ acting on $\ch$. We denote by $\cu(\ch)$ the group of all unitary
operators acting on $\ch$. Clearly, if $U$ is in  $\cu(\ch)$, then $\rho_U$ is an automorphism. \\
In this paper, however,  we will be mostly concerned with isometries and unitaries of a particular form. To define them, we first need
to fix an orthonormal basis, say $\{e_k: k\in\bz\}$, once and for all.
For any
injective map $\Psi$ of $\bz$ to itself, consider the isometry $S_\Psi$ acting on the basis as
$$S_\Psi e_i:=e_{\Psi(i)}\,,\quad i\in\bz\,.$$
By a slight abuse of notation we simply denote by $\rho_\Psi$ the  endomorphism
of ${\rm CCR}(\ch,\sigma)$ corresponding to $S_\Psi$.  Note that
$$\rho_\Psi(W(e_i))=W(e_{\Psi(i)})\,, \quad i \in\bz\,.$$
In particular, corresponding to $\Psi(i)=i+1$, $i\in\bz$, we get the so-called
shift automorphism.\\
We  next establish some notation to introduce the group actions we are interested in.
Let $\bp_\bz$ denote the group of all finite permutations of $\bz$, that  is of all bijective maps
$\s$ of $\bz$ into itself such that $\sigma(i)\neq i$ for finitely many $i$ in $\bz$.
By $\mathbb{U}_\bz$ we denote the subgroup of $\cu(\ch)$ made up of all unitaries
$U$ for which there exists $N>0$ (depending on $U$) such that $Ue_i=e_i$ for all $|i|> N$.
If we denote by $U_{i, j}$ the entries of a given $U$ in $\mathbb{U}_\bz$, {\it i.e.}
$U_{i, j}:= \langle Ue_i, e_j \rangle$, for all $i, j\in\bz$, we have
$U_{i, j}=\delta_{i,j}$ for all $i, j$ with $|i|, |j|> N$.
In particular, it follows that $\mathbb{U}_\bz$ is an inductive limit of unitary matrix groups. Thus,
it is convenient to think of an element $U$ of $\mathbb{U}_\bz$ as an infinite matrix $(U_{i,j})_{i,j\in\bz}$.\\
Lastly, $\bo_\bz$ is the subgroup of  $\mathbb{U}_\bz$ obtained by considering real unitary maps, namely
$O$ in $\mathbb{U}_\bz$ belongs to $\bo_\bz$  if and only if
$\langle Oe_i, e_j \rangle$ are real numbers for all $i, j\in\bz$.
The three groups  introduced all act through automorphisms on ${\rm CCR}(\ch, \s)$
by \eqref{action}. This allows us to consider the  invariant states of
the corresponding  $C^*$-dynamical systems.\\
By $C^*$-dynamical system  we mean here  a triplet $(\ga, G,\alpha)$, where $\ga$ is a unital $C^*$-algebra, $G$ a group, and
$\alpha:G \rightarrow {\rm Aut}(\ga)$ a group homomorphism, {\it i.e.} $\alpha_{gh}=\alpha_g\circ\a_h$, $g, h\in G$.\\
A state $\om$ on $\ga$ is $G$-invariant (or invariant under the action $\a$ of $G$) if $\om\circ\alpha_g=\om$ for every $g\in G$. The  set of all
$G$-invariant states is denoted by $\cs^G(\ga)$.\\
Going back to ${\rm CCR}(\ch, \s)$, its states invariant under $\bp_\bz$ are known as exchangeable (or symmetric) states.
Borrowing the terminology from Classical Probability, we shall henceforth refer to
 $\bo_\bz$-invariant states as  orthogonally invariant or rotatable states. We shall refer to $\mathbb{U}_\bz$-invariant states
as unitarily invariant states.\\
For any $C^*$-dynamical system $(\ga, G, \a)$, $\cs^G(\ga)$ is a weakly* compact convex subset of
the set of all states of $\ga$. The set of extreme points  of $\cs^G(\ga)$, which is non-empty by virtue of the Krein-Milman theorem, is
 denoted by $\ce(\cs^G(\ga))$. The elements of $\ce(\cs^G(\ga))$ are often referred to as the ergodic states of the given
$C^*$-dynamical system.\\
On the GNS space $\ch_\om$ of any $G$-invariant state $\om$, it is possible to define a family of unitaries $\{U_g^\om: g\in G\}$ that implement the action
of the group itself, namely for each $g\in G$
$$U_g^\om \pi_\om(a) (U_g^\om)^*= \pi_\om(\a_g(a))\,, \quad a\in\ga\,.$$
The closed subspace of invariant vectors is then defined as
$$\ch_\om^G:=\bigcap_{g\in G}\{\xi\in\ch_\om: U_g^\om\xi=\xi\}\,.$$
Note that $\ch_\om^G$ is never zero as it always contains $\bc\xi_\om$, where $\xi_\om\in\ch_\om$ is
the GNS vector of $\om$. The orthogonal projection onto $\ch_\om^G$ is denoted by $E_\om$.
We recall that $\ch_\om^G=\bc\xi_\om$ implies that $\om$ is ergodic.
Unlike what happens in the classical theory, the converse implication may fail to hold; it does hold, however, if
we assume that the system is $G$-abelian, that is if for any $\om\in\cs^G(\ga)$, the family of operators
$\{E_\om\pi_\om(a)E_\om: a\in \ga\}$ is commutative, see Theorem 3.1.12 in \cite{S}. Moreover,
$G$-abelianness is often obtained as a consequence of a stronger property known as asymptotical
abelianness: a $C^*$-dynamical system $(\ga, G, \a)$ is asymptotically abelian if there exists a sequence
$\{g_n:n\in\bn\}\subset G$ such that, for every $a, b\in\ga$,
$\lim_{n\rightarrow\infty}\|\a_{g_n}(a)b-b\a_{g_n}(a)\|=0$, see Proposition
3.1.16 in \cite{S}.\\

\section{Invariant States}

We denote by $\ga$ the CCR algebra associated with $\bc$ endowed with its canonical
symplectic form $\sigma(z, w):= \Im z\bar{w}$, $z, w\in \bc$.
As is well known, $\ga$ is a (non-separable) simple nuclear
$C^*$-algebra, see {\it e.g.} \cite{Evans}. To ease the notation, we will from now on denote by $\mathfrak{B}$ the CCR algebra associated with
$(\ch, \sigma)$. Let $\ch_n\subset\ch$ denote the finite-dimensional subspace generated by
$\{e_k: |k|\leq n\}$. We will think of $\ch_n\subset\ch$ as a symplectic real space whose symplectic form is the restriction
of $\s$ to $\ch_n$. The corresponding CCR algebra,  that is
${\rm CCR}(\ch_n, \sigma)$, will be denoted by $\mathfrak{B}_n$. Since $\mathfrak{B}_n\subset\mathfrak{B}_{n+1}$,
we can consider the inductive limit $\mathfrak{B}_0\subset\mathfrak{B}$ of the sequence of $C^*$-algebras $\{\mathfrak{B}_n: n\in\bn\}$. More concretely,
$\mathfrak{B}_0$ is the norm completion in $\gb$ of $\bigcup_n \gb_n$.
Note that $\gb_0$ is itself a CCR algebra, in that $\gb_0={\rm CCR}(\ch_0,\sigma)$ with
$\ch_0=\bigcup_{n}\ch_n$.
The inclusion
$\gb_0\subset\gb$ is proper, as follows from {\it e.g.} \cite[Proposition 5.2.9.]{BR2}.
What is more relevant to our analysis is that the subalgebra $\gb_0$ decomposes as an infinite tensor product
of $\ga$ with itself. Let us quickly recall what the isomorphism looks like.
For each $k\in\bz$, denote by $i_k: \ga\rightarrow \bigotimes_\bz \ga$ the
$k$-th embedding, that is $$i_k(a)=\cdots\otimes 1 \otimes 1\otimes \underbrace{a}_{k^{\rm th}}\otimes 1\otimes 1\otimes\cdots\, .$$
That said, the isomorphism alluded to above is nothing but the map
$\Psi: \gb_0\rightarrow\bigotimes_\bz \ga$ uniquely determined  by
\begin{equation}\label{iso}
\Psi\bigg(W\bigg(\sum_{|k|\leq n} z_k e_k\bigg)\bigg)=\prod_{|k|\leq n} i_k(W(z_k))
\end{equation}
for all $n\in\bn$ and for all $z_k\in\bc$ with $k= -n, \ldots, n$. Elements of $\bigotimes_\bz \ga$ of the form
$i_{k_1}(a_1)\cdots i_{k_n}(a_n)$, for $n\in\bn$, $k_1, \ldots, k_n\in\bz$ and $a_1, \ldots, a_n\in\ga$,
 will be sometimes referred to as
localized tensors.\\
Rather than work in $\gb_0$ only,  we will sometimes use  its isomorphic copy. To do so, we first need
to rewrite the action of $\mathbb{U}_\bz$ on $\bigotimes_\bz\ga$. This is done by combining \eqref{action} and \eqref{iso}, which for  any $U=(U_{i, j})_{i,j\in\bz}$ in $\mathbb{U}_\bz$ gives
\begin{equation}\label{actens}
\rho_U(i_k(W(z)))=\prod_j i_j(W(U_{j, k}z))\,, \quad k\in\bz\,,\,\, z\in\bc\, .
\end{equation}
As an immediate consequence of \eqref{actens}, for every $U\in \mathbb{U}_\bz$, one has $\rho_U(\gb_0)\subseteq\gb_0$.\\
We start our analysis of invariant states on the CCR algebra by considering
the smaller dynamical system obtained by restricting the dynamics to
$\gb_0$. To this aim, we will need to consider the following sequence $\{g_n: n\in\bn\}$ of bijections of
$\bz$:

$$
g_n(k)
=\begin{cases}
 k+ 2^{n-1} \, ,& 0\leq k < 2^{n-1}\\
  k -2^{n-1}\, ,& 2^{n-1}\leq k < 2^{n}\\
  k\, , & k\geq2^{n}\,\\
k- 2^{n-1}\,,&- 2^{n-1}\leq k \leq -1\\
k + 2^{n-1}\,, &- 2^n\leq k < -2^{n-1}\\

k\,, & k< -2^n
\end{cases}
$$
For each $n$, $g_n$ is an element of the group $\bp_\bz$. In particular, the
sequence $\{g_n: n\in \bn\}$ is contained in $\bo_\bz$.
By a slight abuse of notation, we continue to denote by
$g_n$ the corresponding matrix. Note that for each line of this matrix there is exactly one entry that is equal to $1$, while the others are all $0$.

\begin{lem}\label{asyabel}
The restricted $C^*$-dynamical system $(\gb_0, \bo_\bz, \rho)$ is
asymptotically abelian.
\end{lem}

\begin{proof}
We will prove that, for any $a, b \in\bigotimes_\bz\ga$, one has
$$\lim_{n\rightarrow\infty} \|\rho_{g_n}(a)b-b\rho_{g_n}(a)\|=0\, .$$
A standard density argument shows it is enough to deal with $a, b$ lying in a dense subalgebra of
$\bigotimes_\bz\ga$. With this in mind, we can consider localized simple tensors $a, b$. In other words, there is no lack of generality if we take $a=\prod_{|k|\leq m }i_k(W(\lambda_k))$ and
$b=\prod_{|l|\leq m }i_l(W(\mu_l))$ for some $m>0$ and $\lambda_k, \mu_l\in\bc$ for all $k, l\in\{ -m, \ldots,0 ,\ldots, m \}$.
Applying \eqref{actens}, one finds
$$\rho_{g_n}(a)=\prod_j\prod_{|k|\leq m} i_j(W(\lambda_k (g_n)_{j,k})\, .$$
Now if $n$ satisfies $2^{n-1}>m$, the definition
of $g_n$ implies that the entries $(g_n)_{j, k}$ are $0$ for all $j$ with $|j|>m$. Phrased
differently, for $n$ such that $2^{n-1}>m$ one has that $\rho_{g_n}(a)$ is a simple tensor whose $j$-th factors are different from the identity $I$ only when $|j|>m$, hence $\rho_{g_n}(a)$ commutes with $b$.
\end{proof}

As a consequence of the lemma above, we find that $(\gb_0, \bo_\bz, \rho)$ is an $\bo_\bz$-abelian dynamical system.
Moreover, $\bo_\bz$-abelianness in turn implies that $\cs^{\bo_\bz}(\gb_0)$ is a Choquet simplex, see Theorem 3.1.14 in \cite{S}.\\
The next goal we want to achieve is to show that any extreme state of $\cs^{\bo_\bz}(\gb_0)$ factors into an infinite product of a single state $\rho$ of the sample algebra $\ga$. 
To do so, we first need a couple of lemmas.

\begin{lem}\label{Ogn}
For every  localized simple tensor  $a$ in $\gb_0\cong\bigotimes_\bz\ga$ and  every $O$ in $\bo_\bz$, one has
$$\rho_{Og_n}(a)=\rho_{g_n}(a)$$
eventually for all $n$.
\end{lem}

\begin{proof}
By hypothesis $a$ can be taken as $\prod_{|k|\leq m }i_k(W(\lambda_k))$
for some $m>0$ and $\lambda_k\in\bc$ for all $k$'s in $\{-m, \ldots, m\}$.
In addition, there exists $d >0$ such that $O_{i,j}=\delta_{i, j}$ for all $i, j$ with
$|i|, |j|>d$. The statement holds if $n$ satisfies $2^{n-1}\geq\max\{m, d\}$.
Indeed, $\rho_O$ acts as the identity on simple tensors whose factors different from $I$ lie
outside the discrete interval $\{-d, \ldots, d\}$, and $\rho_{g_n}(a)$ is a tensor of this form.
\end{proof}

\begin{lem}\label{prelclust}
If $\om$ is an extreme state in $\cs^{\bo_\bz}(\gb_0)$, then for every $a$ in $\gb_0$ one has
$$\lim_{n\rightarrow\infty}\pi_\om(\rho_{g_n}(a))\xi_\om=\om(a)\xi_\om$$
in the weak topology of $\ch_\om$.
\end{lem}

\begin{proof}
As usual, we will identify $\gb_0$ with the infinite tensor product $\bigotimes_\bz\ga$.
We do not harm the generality of the proof if we suppose that $a$ is localized.
By weak compactness of the unit ball of the Hilbert space $\ch_\om$, the bounded sequence $\{\pi_\om(\rho_{g_n}(a))\xi_\om: n\in\bn\}\subset\ch_\om$
must have at least one accumulation point, say $\xi$.
Now it is easy to verify that any such vector is $\bo_\bz$-invariant. Indeed, by definition there exists a subnet
$\{\pi_\om(\rho_{g_{n_\a}}(a))\xi_\om: \a\in I\}$ such that $\xi=\lim_\a \pi_\om(\rho_{g_{n(\a)}}(a))\xi_\om$ (in the weak
topology of $\ch_\om$). But then  for
any $O\in\bo_\bz$ we have
\begin{align*}
U_O^\om \xi=&\lim_\a U_O^\om\pi_\om(\rho_{g_{n_\a}}(a))\xi_\om=\lim_\a U_O^\om\pi_\om(\rho_{g_{n_\a}}(a))(U_O^\om)^*\xi_\om\\
=&\lim_\a \pi_\om(\rho_{Og_{n_\a}}(a))\xi_\om=\lim_\a\pi_\om(\rho_{g_{n_\a}}(a))\xi_\om=\xi
\end{align*}
where the second-last equality is a consequence of Lemma \ref{Ogn}.
Since $\om$ is extreme, we must have $\xi=\lambda\xi_\om$ for some
$\lambda$ in $\bc$ due to
Lemma \ref{asyabel} and Proposition 3.1.12 in \cite{S}.
Now $\lambda$ is easily seen to equal $\om(a)$ irrespective of the chosen
subsequence. In other words, $\om(a)\xi_\om$ is the only accumulation point of the set $\{\pi_\om(\rho_{g_n}(a))\xi_\om: n\in\bn\}$. Therefore, by weak compactness of the unit ball of $\ch_\om$  the whole
sequence $\{\pi_\om(\rho_{g_n}(a))\xi_\om: n\in\bn\}$ must converge to $\om(a)\xi_\om$, which ends the proof.
\end{proof}

\begin{prop}\label{extreme}
If $\om$ is an extreme state in $\cs^{\bo_\bz}(\gb_0)$, then there exists a state $\varphi$ on $\ga$ such that
$\om=\otimes_\bz  \varphi$.
\end{prop}

\begin{proof}
We start by showing that any extreme state is strongly clustering, that is
$$\lim_n \om (\rho_{g_n}(a) b)=\om(a)\om(b)$$
for all $a, b\in\gb_0$. This is a consequence of Lemma \ref{prelclust}, for
\begin{align*}
\lim_n \om (\rho_{g_n}(a) b)&=\lim_n \langle \pi_\om(\rho_{g_n}(a))\pi_\om(b)\xi_\om ,\xi_\om \rangle\\
&=\lim_n \langle \pi_\om(b)\xi_\om, \pi_\om(\rho_{g_n}(a^*))\xi_\om\rangle=\om(a)\om(b).
\end{align*}
We next define $\varphi$ as a state on $\ga$ by setting
$\varphi(a):=\om(i_j(a))$, $a\in\ga$, $j\in\bz$. The definition is well posed since it does not depend on $j$, for $\om$ is exchangeable being
even rotatable. In order to prove the thesis, we need to show that
$$\om(i_{j_1}(a_1) \cdots i_{j_n}(a_n) )=\varphi(a_1)\cdots\varphi(a_n)$$
for every $n\in\bn$, for every $j_1, \ldots, j_n$ in $\bz$, and $a_1, \ldots, a_n$ in $\ga$. This can be done by induction
on $n$ exactly as in the proof  of \cite[Theorem 2.7.]{StorJFA69}, to which the reader is referred for the missing
details.
\end{proof}

We denote by $\cs^{\bz}(\gb_0)$  the set of stationary states on $\gb_0$, that is of those states that are invariant under the shift automorphism.
In order to state the following result, we recall that
a convex subset $\Delta\subset C$ of a convex set $C$ is a face
if given $x, y\in C$ such that $tx+(1-t)y$ lies in $\Delta$ for some $t$ with $0<t<1$, then $x, y$ are in $\Delta$, in which case we write $\Delta\subset_F  C$.

\begin{prop}\label{faces}
There holds the inclusion of faces:
$$\cs^{\mathbb{U}_\bz}(\gb_0)\subset_F \cs^{\mathbb{O}_\bz}(\gb_0) \subset_F \cs^{\bp_\bz}(\gb_0)\subset_F \cs^{\bz}(\gb_0)\,.$$
\end{prop}

\begin{proof}
We start by pointing out that the four sets in the statement are all Choquet simplices by asymptotical abelianness of the corresponding dynamical systems. Therefore, we can argue as in the proof of Theorem 3.5 in \cite{CDR} and reduce the problem
to showing that any extreme state in one of the first three sets (say $\cs^{\mathbb{U}_\bz}(\gb_0))$ remains extreme in the
 next set to the right ($\cs^{\mathbb{O}_\bz}(\gb_0)$).
But this is trivial for the first two inclusions because extreme $\mathbb{U}_\bz$-invariant/$\mathbb{O}_\bz$-invariant/symmetric states are product states, and product states are extreme in $\cs^{\bp_\bz}(\gb_0)$ as follows from Theorem 2.7
in \cite{StorJFA69}. As for $\cs^{\bp_\bz}(\gb_0)$ being a face of $\cs^{\bz}(\gb_0)$, one can proceed in the same way as done in the proof of Theorem 3.5 in \cite{CDR}.
\end{proof}

Now the state $\varphi\in\cs(\ga)$ whose existence is guaranteed by Proposition \ref{extreme} cannot be chosen arbitrarily.
In fact, rotatability on $\om$ imposes strict constraints on the distribution of $\varphi$ on position (and momentum) operators.
More precisely, the distributions turn out to be Gaussian, as we are going to show.\\
To any state $\varphi\in\cs(\ga)$ we associate the function $G_\varphi: \bc\to\bc$ defined as
$$G_\varphi(z):=\varphi(W(z))\,,\quad z\in\bc\,.$$
Note that $G$ is bounded, $|G_\varphi(z)|\leq 1$ for all $z\in\bc$, and $G_\varphi(0)=1$.
It is worth mentioning that in the literature  $G_\varphi$ is sometimes referred to as the characteristic function of the state
$\varphi$.

\begin{lem}\label{functeq}
Let $\varphi$ be a state on $\ga$. Then the product $\om=\otimes_\bz  \varphi$ is a rotatable state on $\gb_0$ if and only if $G_\varphi$ satisfies the equations
$$G_\varphi(z)=\prod_j G_\varphi(O_{j, k}z)\,,\quad z\in\bc\,,\,\, k\in\bz\,,$$
for any given matrix $O=(O_{k, j})_{k,j\in\bz}$ in $\bo_\bz$.
\end{lem}

\begin{proof}
It is a matter of direct computation. If $\om$ is rotatable, then
thanks to Equation \eqref{actens} we have
\begin{align*}
G_\varphi(z)=&\varphi(W(z))=\om(i_k(W(z))=\om(\rho_O(i_k(W(z)))\\
=&\om\big(\prod_j i_j(W(O_{j, k}z))\big)=\prod_j \varphi(W(O_{j, k}z))\\
=&\prod_j G_\varphi(O_{j, k}z)\,
\end{align*}
for all $z$ in $\bc$.
On the other hand, if the equation is satisfied, then similar calculations show that
$\om$ is rotatable.
\end{proof}

\begin{rem}\label{UZ}
Clearly, both Proposition \ref{extreme} and Lemma \ref{functeq}  hold for $\mathbb{U}_\bz$-invariant states as well since Lemma
\ref{prelclust} applies to extreme  states in $\cs^{\mathbb{U}_\bz}(\gb_0)$.
\end{rem}

The next result spells out in what sense the distribution of position (or momentum) operators in a rotatable state is Gaussian, as
announced.

\begin{prop}\label{gaussian}
Let $\varphi$ be a state on $\ga$. If $\om=\otimes_\bz  \varphi$ is a rotatable state on $\gb_0$, then
either $G_\varphi=\mathbbm{1}_{\{0\}}$, the indicator function of $\{0\}$, or
there exists $\sigma^2\geq 0$ such that
$$G_\varphi(t)=e^{-\frac{\sigma^2 t^2}{2}},\quad  t\in\br\,.$$
\end{prop}

\begin{proof}
Since $G_\varphi$ must satisfy the equations in the statement of Lemma \ref{functeq}, in particular
we find $G_\varphi(z)=G_\varphi(-z)$ for all $z\in\bc$.
As a consequence, the restriction of $G_\varphi$ to the real line is completely determined by its restriction
to the positive half-line. Furthermore, by Proposition 3.1 in \cite{Petz}, the restriction of
$G_\varphi$ to the real line must be a definite-positive function.\\
Let now $z=x+iy$ be any given complex number. By considering a $2\times 2$ orthogonal matrix
given by a rotation by an angle $\a$ the functional equation yields
$$G_\varphi(|z|)=G_\varphi(\cos\a|z|)G_\varphi(\sin\a|z|). $$
If we now  choose $\a$ any argument of $z$, we find $G_\varphi(\sqrt{x^2+y^2})=G_\varphi(x)G_\varphi(y)$.
Henceforth we will denote by $g$ the restriction of $G_\varphi$ to the real line. From what we have seen, $g$ is even and
\begin{equation}\label{eqforg}
g(\sqrt{x^2+y^2})=g(x)g(y)
\end{equation}
for all $x, y\geq 0$.
Since $g$ is even, we may as well work with the restriction of $g$ to the positive half line.
Let us define $h: [0, \infty)\rightarrow\bc$ as $h(t)=g(\sqrt{t})$, $t\in\br$.
From  \eqref{eqforg} one has that $h$ satisfies
$h(t+s)=h(t)h(s)$, for $t, s\geq 0$.
Let $\r$ be the absolute value of $h$, {\it i.e.} $\r(t)= |h(t)|$, $t\geq 0$. We clearly have
$\r(t+s)=\r(t)\r(s)$, for all $t, s\geq 0$.
There are two cases to deal with depending on whether $\r$ may vanish at a point or not.
If it does, that is if there exists $t_0 >0$ with $\r(t_0)=0$, then $\r=\mathbbm{1}_{\{0\}}$.
Indeed, if $t\geq t_0$, then $\r(t)=\r(t-t_0)\r(t_0)=0$. Furthermore, $\r(\frac{t_0}{n})=\r(t_0)^{\frac{1}{n}}=0$ for every
$n\in\bn$ with $n\geq 1$. In conclusion $\r(t)=0$ for all $t>0$, hence $\r=\mathbbm{1}_{\{0\}}$.\\
If $\r$ never vanishes, then there exists $\Psi:[0, \infty)\rightarrow\br$ such that $\r(t)=e^{\Psi(t)}$ and
$\Psi(t+s)=\Psi(t)+\Psi(s)$ for all $t, s\geq 0$. Clearly, any such $\Psi$ is the restriction of a function $\widetilde{\Psi}$ defined on the whole real line that satisfies Cauchy's functional equation, namely $\widetilde{\Psi}(t+s)=\widetilde{\Psi}(t)+\widetilde{\Psi}(s)$, $t, s$ in $\br$. If $\widetilde{\Psi}$ is continuous, then
$\widetilde{\Psi}(t)=ct$, $t\in\br$, for some real constant $c$. In addition, for $G_\varphi$ to be bounded, it is necessary
to take only  non-positive values, say $c=-\frac{\s^2}{2}$ for some $\sigma$ in $\br$.
Discontinuous (non-measurable) solutions of Cauchy's equation are ruled out because their graph is well known to be dense in
$\br^2$, thus giving  rise to an unbounded $\rho$.\\
From now on we can work under the hypothesis that $\r(t)>0$ for all $t\geq 0$, which allows us to decompose $h$ as
$h(t)=\r(t)\chi(t)$, for all $t\geq 0$, with  $\chi(t)\in\bt:=\{z\in\bc: |z|=1\}$ for all $t\geq 0$. Now
$\chi$ satisfies the same equation as $h$, that is $\chi(t+s)=\chi(t)\chi(s)$.
The conclusion will be fully reached once we have shown that $\chi$ must be identically $1$ (when $\r>0$).
To this end, note that $\chi$ cannot take non-real values. Suppose on the contrary there exists
$t_0>0$ such that $\chi(t_0)$ is not real. The function $G_\varphi(t)=\chi(t^2)\r(t^2)$, $t\in\br$, then turns out not to be
positive definite. Indeed, the $2\times 2$ matrix $(G_\varphi(t_i-t_j))_{i,j=1,2}$
with $t_1=0$ and $t_2=t_0$ is
$$
\begin{pmatrix}
1 & \chi(t_0^2)\r(t_0^2)  \\
\chi(t_0^2)\r(t_0^2) & 1
\end{pmatrix}\,,
$$
which even fails to be Hermitian. \\
In other terms, the only possible values for $\chi$ are $1$ and $-1$. But, again, the values $-1$ is ruled out.
Indeed, if there existed $t_0$ with $\chi(t_0)=-1$, then from the very equation satisfied by $\chi$ we would have $\chi(\frac{t_0}{2})=\sqrt{\chi(t_0)}$, and both square roots of $-1$ are not real, being $\pm i$.

\end{proof}

\begin{rem}\label{classical}
The arguments employed in the proposition above  along with Lemma \ref{asyabel} and Proposition
\ref{extreme} provide a novel proof of
Freedman's theorem for commutative random variables. Precisely, the classical theorem is obtained by restricting our dynamics to the commutative
subalgebra $C^*(W(te_i): i\in\bz, t\in\br)\cap \gb_0$, which still factors into a tensor product in that there holds the  $*$-isomorphism
$$C^*(W(te_i): i\in\bz, t\in\br)\cap\gb_0\cong\bigotimes_\bz {\rm CCR}(\br)\, ,$$ 
where ${\rm CCR}(\br)$ is the CCR of $\br$ as a symplectic space with null symplectic form, which can be concretely thought of
the $C^*$-subalgebra generated by the set $\{W(te_1): t\in\br\}$.
It is well known that ${\rm CCR}(\br)$ is isomorphic with ${\rm AP}(\br)$, the $C^*$-algebra  of almost periodic functions on $\br$. Moreover, ${\rm AP}(\br)$ can be identified with
$C(\beta\br)$, the $C^*$-algebra of continuous functions on the Bohr
compactification of the real line, see  {\it e.g.} Chapter 16 in \cite{Dix}.
Working with ${\rm AP}(\br)$ rather than $C_0(\br)$ causes no loss of generality, for the set of states of the latter (Borel probability measures on the real line) naturally injects into
those of the former.


\end{rem}

 Our next aim is to address the CCR algebra so as to say something more
precise about the structure of its rotatable states. To this end, the first thing to do is to exhibit all solutions
of the functional equation considered in Lemma \ref{functeq}.

\begin{lem}\label{allsolutions}
If $G: \mathbb{C}\rightarrow\mathbb{C}$ is a bounded function satisfying
$$G(z)=\prod_j G(O_{j, k}z)\,,\,\, z\in\bc\,, k\in\bz\,,$$
for all orthogonal matrices $O=(O_{j, k})_{k, j\in\bz}$ in $\bo_\bz$, then
$$G(z)=e^{-\sigma^2(\arg z) |z|^2}\,, \,\, z\in\bc\,,$$
where $\sigma^2$ is a positive (possibly infinite) function depending only on the argument of $z$.
\end{lem}

\begin{proof}
It is an application of Proposition \ref{gaussian}, thanks to which the solutions of our equations can be determined by considering
their restrictions to lines through the origin. More precisely, let $z_0=x_0+i y_0$ be a fixed complex number. Let us consider
the line in the complex plane $\{tz_0=tx_0+ity_0: t\in\br\}$ and denote by $g:\br\rightarrow\bc$ the restriction
of $G$ to this line, that is $g(t):= G(tz_0)$, $t\in\br$.
Now the function $g$ can be determined in the same way as in the proof of Proposition \ref{gaussian}, which means we must have
$g(t)=e^{-\sigma^2( z_0) t^2}=e^{-\frac{\sigma^2(z_0)t^2|z_0|^2}{|z_0|^2}}$ for all real $t$.
Setting $z=tz_0$, one finally sees that $G(z)=e^{-{\sigma^2(z)}|z|^2}$, where
$\sigma^2(z)=\frac{\sigma^2(z_0)}{|z_0|^2}$ depends only on the argument of $z=x+iy$, say on the ratio $\frac{y}{x}$.
\end{proof}
We next exhibit a simple class of rotatable states.
\begin{example}\label{rota}
\emph{
For all reals $\lambda, \mu\geq 0$
$$\varphi(W(z)):=e^{-\frac{1}{2}(\lambda x^2+\mu y^2)}\,, \quad z=x+iy\,,$$
defines a state on $\ga$  if and only if $\lambda \mu\geq 1$.
This is a direct application of \cite[Proposition 10]{MV}, which says that for
$\varphi(W(z))=e^{-\frac{1}{2}\alpha(z, z)}$, $z$ in $\bc$, to define a state, it is necessary and sufficient that
$$\alpha(z, z)\alpha(w, w)\geq \sigma(z, w)^2$$
for all $z, w$ in $\bc$, where $\alpha(z, z)=\lambda x^2+\mu y^2$ with
$z=x+iy$.\\
Take $z=x_1+iy_1$ and $w=x_2+iy_2$. If the inequality holds, then
$(\lambda x_1^2+\mu y_1^2)(\lambda x_2^2+\mu y_2^2)\geq(y_1x_2-x_1y_2 )^2$. Choosing
$x_2=0$, we find $(\lambda\mu-1)x_1^2+\mu^2 y_1^2\geq 0$ for all
$x_1, y_1\in\br$, hence $\lambda\mu-1\geq 0$.\\
Conversely, if $\lambda\mu\geq 1$, then
\begin{align*}
\alpha(z, z)\alpha(w, w)&=(\lambda x_1^2+\mu y_1^2)(\lambda x_2^2+\mu y_2^2)\\
&=\lambda^2x_1^2x_2^2+
 \lambda\mu x_1^2y_2^2+\lambda\mu y_1^2x_2^2+\mu^2y_1^2 y_2^2\\
& \geq-2\lambda\mu x_1x_2y_1 y_2 +\lambda\mu x_1^2y_2^2+\lambda\mu y_1^2x_2^2\\
&=
\lambda\mu(x_1y_2- y_1x_2 )^2\\
&\geq (x_1y_2- y_1x_2 )^2=\sigma(z, w)^2\,.
\end{align*}
Accordingly, for any choice of $\lambda, \mu$ with $\lambda\mu\geq 1$, the infinite tensor product $\otimes_\bz\varphi$ provides a rotatable state on $\gb_0$.}
\end{example}

Lemma \ref{allsolutions} allows us to completely determine what $G_\varphi$ looks like when
$\om=\otimes_\bz  \varphi$ is a $\mathbb{U}_\bz$-invariant state. In particular, the rotatable states discussed
in Example \ref{rota} are ruled out unless $\lambda=\mu$.

\begin{prop}\label{unitary}
If $\om=\otimes_\bz  \varphi$ is a $\mathbb{U}_\bz$-invariant state on $\gb_0$, then
either $G_\varphi=\mathbbm{1}_{\{0\}}$ or
$$G_\varphi(z)= e^{-\frac{\sigma^2 |z|^2}{2}}\,, \quad z\in\bc\,,$$
where $\sigma^2$ is a constant with $\sigma^2\geq  1$.
\end{prop}

\begin{proof}
Under the current hypotheses, the function $G_\varphi$ now satisfies the stronger conditions
$G(z)=\prod_j G(U_{j, k}z)\,,\,\, z\in\bc\,, k\in\bz\,,$ for all
$U=(U_{k, j})_{j,k\in\bz}$ in $\mathbb{U}_\bz$.
In particular, $\mathbb{U}_\bz$-invariance also imposes the extra condition $G_\varphi(z)=G_\varphi(e^{i\theta}z)$, $z\in\bc$, which holds for all $\theta\in\br$.
But then $G_\varphi(z)$ only depends on $|z|$, which means it is uniquely determined by its restriction to the real line. Therefore, Proposition \ref{gaussian} gives $G_\varphi=\mathbbm{1}_{\{0\}}$ or $G_\varphi(z)= e^{-\frac{\sigma^2 |z|^2}{2}}$, $z\in\bc$, for some $\s^2>0$.
Finally, the condition $\sigma^2\geq 1$ comes from Example \ref{rota} with $\lambda=\mu=\sigma^2$.
\end{proof}

\begin{rem}\label{hl}
Conversely, for any $\sigma^2\geq 1$, there exists a unique (quasi-free) state $\varphi_{\sigma^2}$  on  $\ga={\rm CCR}(\bc, \sigma)$ such that $\varphi_{\sigma^2}(W(z))= e^{-\frac{\sigma^2 |z|^2}{2}}$, $z\in\bc$, as follows from
Theorem 3.4 in \cite{Petz}. If we now let $\varphi_\infty$ denote the state on $\ga$ defined by $\varphi_\infty(W(z))=\mathbbm{1}_{\{0\}}(z)$, we clearly have that $$\lim_{\s^2\rightarrow \infty}\varphi_{\s^2}=\varphi_{\infty}$$ in the
weak$*$ topology.
\end{rem}

Denote by $\om_{\sigma^2}$ the state on $\gb_0$ which is the infinite product $\otimes_\bz  \varphi_{\sigma^2}$, where
$\varphi_{\sigma^2}$ is the state on the sample algebra $\ga$ defined in the remark above.
Note that the infinite product of $\varphi_\infty$ with itself is the canonical trace $\tau$ of $\gb_0$, that is
$\tau(W(x))=0$ for all $x\in\ch_0$ with $x\neq 0$. For convenience, the canonical trace will be here denoted by
$\om_\infty$.
In the following corollary, $[1, +\infty]$ denotes  the one-point compactification of the half-line $[1, +\infty)$
 thought of as a topological space with respect to the usual
Euclidean topology,
 whereas $\mathcal{E}(\cs^{\mathbb{U}_\bz} (\gb_0))$ is endowed with the weak$^*$ topology inherited from
$\cs(\gb_0)$.
\begin{cor}\label{homeo}
The map $\Psi: [1, +\infty]\rightarrow\mathcal{E}
(\cs^{\mathbb{U}_\bz} (\gb_0))$ given by
$$
\Psi(\sigma^2)=\begin{cases}
 \om_{\s^2} \, ,&{\rm if}\,\,	 \s^2<\infty \\
  \om_\infty\, ,&{\rm if}\,\, \s^2=\infty
\end{cases}
$$
is a homeomorphism of topological spaces. In particular,
$\mathcal{E}(\cs^{\mathbb{U}_\bz} (\gb_0))$ is a compact subset of $\cs^{\mathbb{U}_\bz} (\gb_0)$
\end{cor}

\begin{proof}
The map is well defined on $[1, +\infty]$ thanks to Remark \ref{hl} and Theorem 2.7 in \cite{StorJFA69}.
Injectivity is obvious and surjectivity is a consequence of Remark \ref{UZ}.\\
The continuity of the map and of its inverse is a matter of straightforward computations, which we can safely leave out.
\end{proof}

\begin{prop}\label{Choquet0}
If $\om$ is a $\mathbb{U}_\bz$-invariant state on $\gb_0$, then there exists a unique Borel probability measure
$\mu$ on $[1, +\infty]$ such that
$$\om=\int _{[1, +\infty]} \om_{\s^2}{\rm d}\mu(\sigma^2)\, .$$
More explicitly, for every $x$ in $\ch_0$ one has
$$\om(W(x))=\int _1^{+\infty} e^{-\frac{\sigma^2\|x\|^2}{2}}{\rm d}\mu(\sigma^2) + \mu(\{\infty\})\om_\infty(W(x))\, .$$
\end{prop}

\begin{proof}
The first formula is a direct consequence of $\cs^{\mathbb{U}_\bz} (\gb_0)$ being a Choquet simplex along with
Corollary \ref{homeo}. The second follows from $\om_{\s^2}(W(x))= e^{-\frac{\s^2\|x\|^2}{2}}$ for $x$ in $\ch_0$, which
can be seen by direct computation. Indeed, for $x=\sum_{i\in F} \lambda_i e_i$, where $F\subset \bz$ is  a finite subset and
$\lambda_i$ is a complex number for every $i\in F$, we  have
\begin{align*}
\om_{\s^2}(W(x))&=\om_{\s^2}\left(\prod_{i\in F} W(\lambda_i e_i)\right)=\prod_{i\in F}\varphi_{\s^2}(W(\lambda_i))\\
&=
\prod_{i\in F} e^{-\frac{\sigma^2 |\lambda_i|^2}{2}}
=e^{-\frac{\sigma^2 \|x\|^2}{2}}\, .
\end{align*}
\end{proof}
The next result shows that most of $\mathbb{U}_\bz$-invariant states  on $\gb_0$ are regular.
\begin{prop}\label{partialreg}
If $\mu$  is a Borel probability measure on $[1, +\infty]$ with $\mu(\{\infty\})=0$, then  the corresponding
$\mathbb{U}_\bz$-invariant state $\om:=\int _1^ {+\infty} \om_{\s^2}{\rm d}\mu(\sigma^2)$ is regular.
\end{prop}

\begin{proof}
In light of Proposition \ref{Choquet0}, for any $x$ in $\ch_0$ we have
$$\om(W(tx))=\int _1^{+\infty} e^{-\frac{\sigma^2t^2\|x\|^2}{2}}{\rm d}\mu(\sigma^2) $$
The continuity in $t$ of the above expression is a consequence of
Lebesgue's dominated convergence theorem since $e^{-\frac{\sigma^2t^2\|x\|^2}{2}}\leq 1$ for all
$t\in\br$ and $\sigma^2\in [1, +\infty)$.
\end{proof}
\noindent
The proof of the above result actually shows more. Indeed, the function $\br\ni t \mapsto \om(W(tx))\in\bc$
is even analytic under the assumption $\mu(\{\infty\})=0$.
\begin{rem}
If $\rho$ is any non-regular state on $\ga= {\rm CCR}(\bc, \sigma)$, then $\otimes_\bz \rho$ is certainly
an extreme symmetric state on $\gb_0$, which by construction fails to be regular. In this way it is possible to exhibit a good many non-regular extreme symmetric states, in stark contrast with $\mathbb{U}_\bz$-invariant states, whose only
non-regular extreme state is the canonical trace.
\end{rem}

By a slight abuse of notation, for any $\s^2\geq 1$ (including $\infty$) we continue to denote by
$\om_{\s^2}$ the quasi-free state on  the whole $\gb$ uniquely determined by $\om_{\s^2}(W(x))= e^{-\frac{\s^2 \|x\|^2}{2}}$,
$x\in\ch$.\\

We are now ready to state our main result.

\begin{thm}\label{Choquet1}
The compact convex $\cs^{\mathcal{U}(\ch)}(\gb)$ is a Bauer simplex whose extreme point set is
$\{\om_{\s^2}: \s^2\in [1, +\infty]\}$. Explicitly, if $\om$ is a $\cu(\ch)$-invariant state on $\gb$, then there exists a unique Borel probability measure
$\mu$ on $[1, +\infty]$ such that
$$\om=\int _{[1, +\infty]} \om_{\s^2}{\rm d}\mu(\sigma^2)\, .$$
\end{thm}

\begin{proof}
We start by observing that if $\om$ is a $\cu(\ch)$-invariant state on $\gb$, its restriction $\om\upharpoonright_{\gb_0}$
to $\gb_0$ is certainly a $\mathbb{U}_\bz$-invariant state on $\gb_0$.
Let us then consider the restriction map $\Psi: \cs^{\cu(\ch)}(\gb)\rightarrow\cs^{\mathbb{U}_\bz}(\gb_0)$ defined as
$\Psi(\om):=\om\upharpoonright_{\gb_0} $, $\om$ in $\cs^{\cu(\ch)}(\gb)$.
We aim to show that $\Psi$ is an affine homeomorphism between compact convex sets.
Injectivity of $\Psi$  follows from transitivity of $\cu(\ch)$ on $\ch$: if $x$ lies in $\ch\setminus\ch_0$, it is possible to find
$y\in\ch_0$ and $U\in\cu(\ch)$ such that $\|x\|= \|y\|$ and $y=Ux$. But then we must have
$\om(W(x))=\om(W(Ux))=\om (W(y))$. In other words, $\om$ is completely determined by its restriction to
$\gb_0$. The map $\Psi$ is clearly surjective thanks to Proposition \ref{Choquet0}.
The continuity of $\Psi$ is trivial. By compactness of $\cs^{\cu(\ch)}(\gb)$  it follows that $\Psi$ is a closed map, which means
its inverse is continuous as well.
\end{proof}

\begin{rem}\label{flat}
The invariance under $\mathbb{U}_\bz$ for states on the whole CCR algebra $\gb$ is not enough
to obtain the thesis of Theorem \ref{Choquet1}. This can be seen by looking
at  the so-called {\it flat} extensions of a state defined on $\gb_0$.
To this end, pick any $\mathbb{U}_\bz$-invariant state $\om$ on $\gb_0$.  Then one can define
a state $\Om$ on the whole $\gb$ by setting
$$
\Om(W(x))
=\begin{cases}
 \om(W(x)) \, ,& x\in\ch_0\\
  0\, ,& x\in\ch\setminus\ch_0\, .
  \end{cases}
$$
The proof that the definition above yields a state on $\gb$ is to be found in \cite[Proposition 2.1]{Ac}.
However, it remains to show that $\Om$ is invariant under the action of $\mathbb{U}_\bz$ on $\gb$.
But this easily follows from the inclusions $U\ch_0\subset\ch_0$ and $U(\ch\setminus\ch_0)\subset\ch\setminus\ch_0$,
which hold true for any unitary $U$ in $\mathbb{U}_\bz$.
Phrased differently, the map that restricts $\mathbb{U}_\bz$-invariant states on $\gb$ to $\gb_0$
fails to be injective.
\end{rem}

In the next proposition we show that the states focused on  in the above result belong to the \emph{folium}
of all normal states in the Fock representation as long as they come from a measure $\mu$ that assigns $0$ to
the point at infinity.

\begin{prop}\label{folium}
Any regular $\cu(\ch)$-invariant state on $\gb$ is normal in the Fock representation.
\end{prop}

\begin{proof}
Thanks to Theorem \ref{Choquet1} any such state is of the form
$$\om=\int _{[1, +\infty]} \om_{\s^2}{\rm d}\mu(\sigma^2)\,,$$
for a suitable probability measure on $[1, +\infty]$ with $\mu(\{\infty\})=0$.\\
We recall that for each $\sigma^2\geq 1$, the state
$\om_{\sigma^2}$ is normal in the Fock representation $\pi_\mathfrak{F}$, which means there exists a trace-class operator
$T_{\sigma^2}$ in $\mathfrak{F}_+(\ch)$ such that
$\om_{\sigma^2}(a)={\rm Tr}(\pi_\mathfrak{F}(a)T_{\sigma^2})$, $a\in\gb$.
Moreover, the operator $T_{\sigma^2}$  can be seen to equal $\dfrac{\Gamma(A(I+A)^{-1})}{ {\rm Tr}(\Gamma(A(I+A)^{-1}))}$ with $A=\frac{\sigma^2-1}{2}I$, {\it cf.} \cite{PP} and
Lemma 6.4.1 in \cite{Pitrik}. That said, we can easily come to the conclusion.
Note that
$$\widetilde\om(S):=\int _1^{+\infty} {\rm Tr}(ST_{\sigma^2}){\rm d}\mu(\sigma^2)\,, \quad S\in\mathcal{B}(\mathfrak{F}_+(\ch))$$
defines a state on $\mathcal{B}(\mathfrak{F}_+(\ch))$ such that $\om(a)=\widetilde\om(\pi_\mathfrak{F}(a))$ for every
$a\in\gb$.
The thesis will be achieved once we prove that $\widetilde\om$ is normal. Let
$\{S_n: n\in\bn\}\subset \mathcal{B}(\mathfrak{F}_+(\ch))$ be a bounded increasing
sequence of positive operator, and let $S$ be its sup. As is known, $S$ is the strong limit of $S_n$, hence for every
$\sigma^2\geq 1$ we have
$\lim_n{\rm Tr}(S_n T_{\sigma^2})={\rm Tr}(S T_{\sigma^2})$. Finally, by the monotone convergence theorem
we find
$$\lim_n\widetilde\om(S_n)=\lim_n\int _1^{+\infty}{\rm Tr}(S_n T_{\sigma^2}){\rm d}\mu(\sigma^2)=\int _1^{+\infty}{\rm Tr}(S T_{\sigma^2}){\rm d}\mu(\sigma^2)=\widetilde\om(S)\,,$$
and we are done.
\end{proof}

We would like to remark that a $\mathcal{U}(\ch)$-invariant state $\om$ on $\gb$
be thought of as a quantum stochastic process of a particular kind. Indeed, any such state is uniquely
determined by its restriction $\om_0$ to $\gb_0$, and the tensor product structure of $\gb_0$ allows us
to recover a stochastic process with sample algebra $\ga={\rm CCR}(\bc, \s)$ from $\om_0$, see  {\it e.g.}  Section 3 of \cite{CRZbis}.
In particular, one may consider the tail algebra of the process, as defined as in {\it e.g.} \cite{CrFid, CRAMPA}, and ask what it looks like. The question is very easily answered for extreme states.
Precisely, if $\om_0$ is an extreme state in any of the sets
$\cs^{\mathbb{U}_\bz}(\gb_0)$, $ \cs^{\mathbb{O}_\bz}(\gb_0)$ or $\cs^{\bp_\bz}(\gb_0)$, then
its tail algebra is trivial. This is a consequence of Proposition 2.9 in \cite{CRZbis}, which addresses extreme symmetric
states (in the wider context of local action on quasi-local $C^*$-algebras). In fact,  extreme states of both $\cs^{\mathbb{U}_\bz}(\gb_0)$ and $ \cs^{\mathbb{O}_\bz}(\gb_0)$ are all instances of extreme symmetric states, being product states.
The tail algebra of non-extreme states can of course fail to be trivial. Nevertheless, it will always be commutative,
{\it cf.} Proposition 2.10 in \cite{CRZbis} and the references therein. Furthermore, there is only one conditional expectation from $\pi_{\om_0}(\gb_0)''$
onto the tail algebra, and, more importantly, the process itself is conditionally independent and identically distributed (with
centered Gaussian distribution for states in $\cs^{\mathbb{U}_\bz}(\gb_0)$ or $ \cs^{\mathbb{O}_\bz}(\gb_0)$)
with respect to the conditional expectation.\\

\section*{Acknowledgments}
\noindent
We acknowledge  the support of Italian INDAM-GNAMPA.\\
The authors are partially supported by Progetto GNAMPA 2023 CUP E55F22000270001 ``Metodi di Algebre di Operatori in Probabilit\`a non Commutativa"   and by Italian PNRR Partenariato Esteso PE4, NQSTI.

\end{document}